\documentclass[11pt]{amsart}
\usepackage{graphicx}
\usepackage{pgf,tikz,pgfplots}
\usepackage{mathrsfs}
\usepackage{mathpple}
\usepackage{tikz-cd}
\usepackage{amsmath}
\usepackage{tikz}
\usepackage{mathdots}

\usepackage{cancel}
\usepackage{color}
\usepackage{siunitx}
\usepackage{array}
\usepackage{multirow}
\usepackage{amssymb}
\usepackage{gensymb}
\usepackage{tabularx}
\usepackage{booktabs}
\usetikzlibrary{fadings}
\usetikzlibrary{patterns}
\usetikzlibrary{shadows.blur}
\usetikzlibrary{shapes}

\usepackage{hyperref,xcolor}% http://ctan.org/pkg/{hyperref,xcolor}
\definecolor{wine-stain}{rgb}{0.5,0,0}
\hypersetup{
  colorlinks,
  linkcolor=wine-stain,
  linktoc=all
}

\usepackage{latexsym,bm,amsmath,amssymb}

\usetikzlibrary{arrows}
\newtheorem{thm}{Theorem}[section]

\newtheorem{lem}[thm]{Lemma}
\newtheorem{prop}[thm]{Proposition}
\theoremstyle{definition}

\newtheorem{defn}[thm]{Definition}

\theoremstyle{remark}
\newtheorem{rem}[thm]{Remark}
\numberwithin{equation}{section}
\DeclareFontFamily{U}{eus}{\skewchar\font'60}
\DeclareFontShape{U}{eus}{m}{n}{%
	<-6>eusm5%
	<6-8>eusm7%
	<8->eusm10%
}{}

\DeclareMathAlphabet\EuScript{U}{eus}{m}{n}

\begin{document}

\title[]{Quadratic duality for chiral algebras}%
\author{Zhengping Gui, Si Li and Keyou Zeng}

  \address{
Z. Gui: International Centre for Theoretical Physics, Trieste, Italy;
}
\email{zgui@ictp.it}
  \address{
S. Li:
Yau Mathematical Sciences Center, Tsinghua University, Beijing, China
}

\email{sili@mail.tsinghua.edu.cn}

  \address{
	K. Zeng:
	Perimeter Institute for Theoretical Physics, Waterloo, Canada
}

\email{kzeng@perimeterinstitute.ca}

\thanks{}%

%\date{}%
%\dedicatory{}%
%\commby{}%
% ----------------------------------------------------------------
\begin{abstract}
We introduce a notion of quadratic duality for chiral algebras. This can be viewed as a chiral version of the usual quadratic duality for quadratic associative algebras. We study the relationship between this duality notion and the Maurer-Cartan equations for chiral algebras, which turns out to be parallel to the associative algebra case. We also present some explicit examples.
\end{abstract}
\maketitle
\tableofcontents

% ----------------------------------------------------------------

\section{Introduction}
The notion of Koszul duality is ubiquitous in mathematics. It has found many applications in representation theory, algebraic geometry, homological algebra, and topology (see, e.g., \cite{beilinson1996koszul,bernshtein1978algebraic,ginzburg1994koszul,goresky1998equivariant}). In the original context of quadratic algebra, it can be understood as the quadratic duality studied by Stewart Priddy \cite{priddy1970koszul}. It is natural to extend this duality to other algebraic structures. In \cite{beilinson2004chiral}, Beilinson and Drinfeld introduce the notion of chiral algebras which encodes the algebraic structure of the chiral part of two-dimensional conformal field theories on arbitrary smooth Riemann surfaces. In some situations, chiral algebras can be viewed as global versions of vertex algebras. For example, it is shown in \cite{frenkel2004vertex} that any quasiconformal vertex algebra gives rise to a chiral algebra on an arbitrary Riemann surface. According to \cite{beilinson2004chiral}, chiral algebras are meromorphic versions of associative algebras. The concept of quadratic algebra in the world of associative algebras can be extended to chiral algebras. The chiral algebra freely generated by a given non-empty set of sections subject to quadratic relations is constructed in \cite{beilinson2004chiral}. The vertex algebra version of this was constructed earlier in \cite{roitman2001combinatorics}.  In this paper, we formulate the notion of quadratic duality for quadratic chiral algebras.

Recall that a quadratic algebra is the quotient algebra $A=T(V)/(R)$ of the tensor algebra of a finite-dimensional graded vector space $V$ by the ideal generated by quadratic relations $R\subset V^{\otimes 2}$. The algebra $A$ is called the quadratic algebra associated to the quadratic datum $(V,R)$. The quadratic dual algebra is defined by the quadratic datum $(s^{-1}V^{\vee},R^{\perp})$, where $V^{\vee}$ is the dual of $V$ and $R^{\perp}\subset (s^{-1}V^{\vee})^{\otimes 2}$ is the annihilator of $R$. Here the notation $s$ denotes the suspension. In the literature, people sometimes use the unsuspended version $(V^{\vee},s^2R^{\perp})$ as the dual algebra.

Following \cite{beilinson2004chiral}, in the context of chiral algebras, the space of generators $V$ is replaced with a graded locally free sheaf $N$ of finite rank on a smooth complex algebraic curve $X$. The space of quadratic relations $R\subset V^{\otimes 2}$ is replaced with a locally free subsheaf $P\subset j_*j^*N\boxtimes N$ such that $P|_U=N\boxtimes N|_{U}$. Here $U=X^2-\Delta$ is the complement of the diagonal and $j:U\hookrightarrow X^2$ is the open embedding. This is the major difference from the quadratic algebra case, since there we allow arbitrary quadratic relations $R\subset V^{\otimes 2}$. It is not surprising because we have the locality in the axioms of chiral algebra (corresponds to the locality in the definition of vertex algebra). By locality, for any local section $a\boxtimes b\in N\boxtimes N$ one can find an integer $k_{a,b}$ sufficiently large such that the chiral operation annihilates $(z_1-z_2)^{k_{a,b}}a\boxtimes b$. This forces us to have the condition $P|_U=N\boxtimes N|_U$. In this case, we say that $(N,P)$ is a chiral quadratic datum. Beilinson and Drinfeld construct the free chiral algebra associated to $(N,P)$ which we denote by $\mathcal{A}(N,P)$. Then we proceed as in the case of quadratic algebras. We take $s^{-1}N^{\vee}_{\omega^{-1}}:=s^{-1}N^{\vee}\otimes _{\mathcal{O}_X}\omega_X^{-1}$ to be the sheaf of dual generators and $P^{\perp}$ to be the sheaf of the chiral annihilator of $P$. More precisely, the subsheaf  $P^{\perp}\subset j_*j^*s^{-1}N^{\vee}_{\omega^{-1}}\boxtimes s^{-1}N^{\vee}_{\omega^{-1}}$ satisfies
$$
\mu(\langle P\otimes \omega_{X^2},P^{\perp}\otimes s^{2}\omega_{X^2}\rangle)=0,
$$
where $\mu:j_*j^*\omega_{X^2}\rightarrow\Delta_*\omega_X$ is the unit chiral operation. We come to the following definition.

\begin{defn}[= Definition \ref{Def_quad_dual}]
 A chiral quadratic datum $(N,P)$ is called dualizable if the pair  $(s^{-1}N^{\vee}_{\omega^{-1}},P^{\perp})$ is also a chiral quadratic datum. In this case, the quadratic dual chiral algebra of $\mathcal{A}=\mathcal{A}(N,P)$ is defined to be $\mathcal{A}^!:=\mathcal{A}(s^{-1}N^{\vee}_{\omega^{-1}},P^{\perp})$.
\end{defn}
%Let $X$ be a smooth complex algebraic curve. A pair $(N,P)$ is called a chiral quadratic datum if $N$ is graded locally free sheaf of finite rank on $X$ and $P\subset j_*j^*N\boxtimes N$ is a locally free subsheaf such that $P|_U=N\boxtimes N|_U$.
Within this set-up, we prove that a morphism from $\mathcal{A}(N,P)$ to an arbitrary graded chiral algebra $\mathcal{B}$ can be viewed as a solution of the Maurer-Cartan equation in the tensor product $\mathcal{A}(s^{-1}N^{\vee}_{\omega^{-1}},P^{\perp})\otimes \mathcal{B}$.

\begin{thm}[= Theorem \ref{Representable1}]
Let $\mathcal{B}$ be a graded chiral algebra. Take a quadratic chiral algebra $\mathcal{A}=\mathcal{A}(N,P)$ and its dual $\mathcal{A}^!=\mathcal{A}(s^{-1}N^{\vee}_{\omega^{-1}},P^{\perp})$.  There exists an injective map
$$
\mathrm{Hom}(\mathcal{A},\mathcal{B})\hookrightarrow \mathrm{MC}(\mathcal{A}^!\otimes\mathcal{B}).
$$
\end{thm}
In some special cases, the above injection can be shown to be a bijection (see Theorem \ref{Representable2}).  To obtain general results parallel to those in associative algebras, one needs to understand the meaning of Koszulness in chiral algebras and work in a suitable homotopy setting. These are not included here and will be investigated in future work.

In the literature, people also consider the non-homogeneous quadratic duality for algebras, see \cite{loday2012algebraic, positsel1993nonhomogeneous, priddy1970koszul}. This can be also extended to the context of chiral algebras (see Section \ref{Nonhomogeneous1}). We can prove similar theorems (see the second part of Section \ref{MCEquation}).

Our work is also motivated by a recent mathematical understanding of AdS/CFT correspondence, or holographic correspondence in physics initiated by Costello and Li in \cite{Costello:2017fbo, Costello:2016mgj}. It turns out that Koszul duality plays a crucial role in the holographic correspondence. Along these lines, \cite{costello2021twisted} proposed a physical definition of Koszul duality for chiral algebras from consideration of defect. They defined the Koszul dual $\mathcal{A}^!$ to be the universal defect chiral algebra that can be coupled to the field theory corresponding to the original chiral algebra $\mathcal{A}$. This physical definition parallels the mathematical result of the bijection between $\mathrm{Hom}(\mathcal{A},\mathcal{B})$ and $ \mathrm{MC}(\mathcal{A}^!\otimes\mathcal{B})$. We refer to \cite{Paquette:2021cij} for an introduction to Koszul duality in physics.

We should emphasize that the chiral Koszul duality considered in \cite{francis2012chiral} is not what we are studying in the present paper. In \cite{francis2012chiral}, they establish an equivalence between chiral Lie algebras  (which we call chiral algebras) and factorization coalgebras (which we call factorization algebras) on higher-dimensional varieties in a form of Koszul duality. In other words, Koszul duality studied there is a way of interchanging two different descriptions of the same underlying algebraic structure. The chiral quadratic duality we study turns a chiral algebra into a new chiral algebra.  However, they mention that the chiral quadratic duality considered here should exist    as an algebro-geometric analogue of $\mathcal{E}_n$-Koszul duality (see \cite[pp15]{francis2012chiral}). For $\mathcal{E}_n$-Koszul duality, we refer readers to \cite{ayala2015factorization}.

Here is an outline of the paper. In Section \ref{ChiralFact} we review chiral algebras and factorization algebras. Section \ref{ChiralQuadratic} is devoted to the quadratic duality theory for quadratic chiral algebras. We also extend the duality to the non-homogeneous case. In Section \ref{MCEquation} we establish a simple relation between chiral quadratic duality and solutions of Maurer-Cartan equations. After these general discussions, in Section \ref{Examples} we present some examples.

\section{Conventions}

\begin{itemize}

\item A graded vector space $V$ is a direct sum of vector spaces $V=\mathop{\bigoplus}\limits_{i\in\mathbb{Z}}V_i$. We will use the $k$-th shift notation $V[k],k\in \mathbb{Z}$ as well as the suspension notation $s^{-k}V,k\in\mathbb{Z}$
$$
(s^{-k}V)_i=(V[k])_i=V_{i+k}
.$$
We also use $|a|$ for the degree of a homogeneous element $a\in V.$
    \item An $\mathcal{O}_X$-module $N$ is called $\mathbb{Z}$-graded if $N=\mathop{\bigoplus}\limits_{i\in \mathbb{Z}}N_i$ where each $N_i$ is an $\mathcal{O}_X$-module. A submodule of a $\mathbb{Z}$-graded $\mathcal{O}_X$-module $N$ means a $\mathbb{Z}$-graded $\mathcal{O}_X$-module $N'=\mathop{\bigoplus}\limits_{i\in \mathbb{Z}}N'_i$ such that $N'_i\subset N_i, \forall i\in\mathbb{Z}$.
    \item We use $X$ for a smooth complex algebraic curve throughout this paper. Let $N$ be a locally free sheaf on $X$. We denote the dual of $N$ by $N^{\vee}$. For abbreviation, we use $N_{\omega^{\pm 1}}:=N\otimes_{\mathcal{O}_{X}} \omega^{\pm 1}_X$, where $\omega_X$ is the canonical sheaf of $X$.
    \item For right $\mathcal{D}_X$-modules $M$ and $N$, the tensor product is $M\otimes N:=(M_{\omega^-1}\otimes_{\mathcal{O}_X}N_{\omega^{-1}})_{\omega}$.
\end{itemize}
\noindent\textbf{Acknowledgment.} The authors would like to thank Asem Abdelraouf, Kevin Costello, Lothar G\"{o}ttsche and Vadym Kurylenko for helpful communications. S. L. is supported by the National Key R\&D Program of China (NO. 2020YFA0713000). K. Z. is supported by the Perimeter Institute for Theoretical Physics. Research at Perimeter Institute is supported in part by the Government of Canada,
through the Department of Innovation, Science and Economic Development Canada,
and by the Province of Ontario, through the Ministry of Colleges and Universities.

\section{Chiral algebras and Factorization algebras}\label{ChiralFact}
In this section, we give a brief review of some basic definitions of chiral algebras and factorization algebras. For more details, see \cite{beilinson2004chiral,frenkel2004vertex,gaitsgory1998notes}. Throughout this section, $X$ stands for a smooth curve over $\mathbb{C}$.

\begin{defn}\label{chiralDefn}
Let $\mathcal{A}$ be a $\mathbb{Z}$-graded $\mathcal{D}_X$-module. A chiral algebra structure on $\mathcal{A}$ is a degree 0 $\mathcal{D}_{X^2}$-module map:
$$
\mu:j_*j^*\mathcal{A}\boxtimes\mathcal{A}\rightarrow \Delta_*(\mathcal{A}),
$$
where $\Delta:X\rightarrow X^2$ is the diagonal embedding and $j:U=X^2-\Delta\hookrightarrow X^2$ is the open embedding.

The map $\mu$ satisfies the following two conditions:
\begin{itemize}
  \item Antisymmetry:

  If $f(z_1,z_2)\cdot a\boxtimes b$ is a local section of $j_*j^*\mathcal{A}\boxtimes\mathcal{A}$, then
  \begin{equation}\label{Antisymmetry}
  \mu(f(z_1,z_2)\cdot a\boxtimes b)=-(-1)^{|a||b|}\sigma_{1,2}\mu(f(z_2,z_1)\cdot b\boxtimes a),
  \end{equation}
  where $\sigma_{1,2}$ acts on $\Delta_*\mathcal{A}$ by permuting two factors of $X^2$.

  \item Jacobi identity:

  If $a\boxtimes b\boxtimes c\cdot f(z_1,z_2,z_3)$ is a local section of $j_*j^*\mathcal{A}^{\boxtimes 3}$ where $j$ is the open embedding of the complement of the big diagonal in $X^3$. Then
  \iffalse
  $$
0=  \mu(\mu(f(z_1,z_2,z_3)\cdot a\boxtimes b)\boxtimes c)+(-1)^{p(a)\cdot (p(b)+p(c))}\sigma_{1,2,3}\mu(\mu(f(z_2,z_3,z_1)\cdot b\boxtimes c )\boxtimes a)+(-1)^{p(c)\cdot (p(a)+p(b))}\sigma_{1,2,3}\mu(\mu(f(z_3,z_1,z_2)\cdot c\boxtimes a )\boxtimes b),
  $$
  \fi
  \begin{align*}
      \mu(\mu(f(z_1,z_2,z_3)\cdot &a\boxtimes b)\boxtimes c)+(-1)^{|a| (|b|+|c|)}\sigma_{1,2,3}\mu(\mu(f(z_2,z_3,z_1)\cdot b\boxtimes c )\boxtimes a)+  \\
     & (-1)^{|c| (|a|+|b|)}\sigma_{1,2,3}^{-1}\mu(\mu(f(z_3,z_1,z_2)\cdot c\boxtimes a )\boxtimes b)=0,
  \end{align*}
  here $\sigma_{1,2,3}$ denotes the cyclic permutation action on $\Delta^{X\rightarrow X^3}_*\mathcal{A}$ and $\Delta^{X\rightarrow X^3}:X\rightarrow X^3\ \text{is the diagonal embedding}$.
  \iffalse
  $=\mathcal{A}\otimes_{\mathcal{D}_X}\mathcal{O}_X\otimes_{\mathcal{O}_{X^3}}\mathcal{D}_{X^3}.$
  \fi
\end{itemize}
\end{defn}

Now we will define factorization algebra. Later we will discuss the relationship between chiral algebras and factorization algebras.

We use the following conventions in the definition. For a surjective map $\pi: J \twoheadrightarrow I$ between two finite sets $I$ and $J$, let $j^{[J / I]}: U^{[J / I]} \hookrightarrow X^J$ be the complement to all the diagonals that are transversal to $\Delta^{(J / I)}: X^I \hookrightarrow X^J$. Therefore one has $$U^{[J / I]}=\left\{\left(x_j\right) \in X^J: x_{j_1} \neq x_{j_2}\  \text{if} \ \pi\left(j_1\right) \neq \pi\left(j_2\right)\right\}.$$
\begin{defn}
A \textit{factorization algebra} on $X$ consists of the following datum:

(1) A graded quasicoherent sheaf $B_{X^I}$ over $X^I$ for any finite set $I$, which has no non-zero local sections supported at the union of all partial diagonals.

(2) Isomorphisms of graded  quasicoherent sheaves
$$
\nu^{(\pi)}
=\nu^{(J/I)}:\Delta^{(\pi)*}B_{X^J}\xrightarrow{\sim} B_{X^I}$$
for every surjection $\pi :J\twoheadrightarrow I$ and compatible with the composition of the $\pi'$s.

(3) (\textit{factorization}) For every surjection $J\twoheadrightarrow I$, there is an isomorphism of  $\mathcal{O}_{U^{[J/I]}}$-modules
$$
c_{[J/I]}:j^{[J/I]*}(\mathop{\boxtimes}\limits_{i\in I}B_{X^{J_i}})\xrightarrow{\sim} j^{[J/I]*}B_{X^J}.
$$
We require that $c$'s are mutually compatible: for $K\twoheadrightarrow J$ the isomorphism $c_{[K/J]}$ coincides with the composition $c_{[K/I]}(\mathop{\boxtimes}\limits_{i\in I} c_{[K_i/J_i]})$. And $c$ should be compatible with $\nu$: for every $J\twoheadrightarrow J'\twoheadrightarrow I$ one has
$$
\nu^{(J/J')}\Delta^{(J/J')*}(c_{[J/I]})=c_{[J'/I]}(\mathop{\boxtimes}\limits_{i\in I} \nu^{(J_i/J'_i)}).
$$

(4)(\textit{unit}) There exists a global section $1$ of $B_X$ such that for every $f\in B_X$ one has $1\boxtimes f\in B_{X^2}\subset j_*j^*B^{\boxtimes 2}_{X}$ and $\Delta^*(1\boxtimes f)=f$.

\end{defn}

There is an equivalence between the category of factorization algebras and that of chiral algebras \cite{beilinson2004chiral}. More precisely, we can obtain a chiral algebra from a factorization algebra $B$ as follows. For each surjection $J\twoheadrightarrow I$ we have a natural isomorphism of left $\mathcal{D}-$modules
$$
\Delta^{(J/I)*}B_{X^J}\xrightarrow{\sim} B_{X^I}.
$$
We can rewrite it as an isomorphism of right $\mathcal{D}_{X^J}$-modules
$$
\Delta^{(J/I)}_*\omega^{\boxtimes I}_X\otimes_{\mathcal{O}_{X^J}} B_{X^J}\xrightarrow{\sim}\Delta^{(J/I)}_*(\omega_X^{\boxtimes I}\otimes_{\mathcal{O}_{X^I}} B_{X^I}).
$$
In particular for $\Delta:X\hookrightarrow X^2$, we have
$$
\Delta_*\omega_X\otimes_{\mathcal{O}_{X^2}} B_{X^2}\xrightarrow{\sim} \Delta_*(\omega_X\otimes_{\mathcal{O}_{X}} B_X).
$$
Then we have
\begin{equation}\label{ChiralOperation}
  j_*j^*B^{r\boxtimes 2}=j_*j^*\omega^{\boxtimes 2}_X\otimes_{\mathcal{O}_{X^2}} B_{X^2}\rightarrow \Delta_*\omega_X\otimes_{\mathcal{O}_{X^2}} B_{X^2}=\Delta_*(\omega_X\otimes_{\mathcal{O}_{X}} B)=\Delta_*B^r.
\end{equation}
\iffalse
$$
j_*j^*B^{r\boxtimes 2}=j_*j^*\omega^{\boxtimes 2}_X\otimes_{\mathcal{O}_{X^2}} B_{X^2}\rightarrow \Delta_*\omega_X\otimes_{\mathcal{O}_{X^2}} B_{X^2}=\Delta_*(\omega_X\otimes_{\mathcal{O}_{X}} B)=\Delta_*B^r.
$$
\fi
One can verify that the above binary operation makes the right $\mathcal{D}_X$-module $B^r$ into a chiral algebra.

Now we explain the inverse direction. Suppose we have a chiral algebra $\mathcal{A}$, then we define $\mathcal{F}_{X^I}=\mathcal{A}^l_{X^I}:=\mathcal{A}_{X^I}\otimes _{\mathcal{O}_{X^I}}\omega_{X^I}^{-1}$ on $X^I$. Here $\mathcal{A}_{X^I}$ is the intersections of the kernels of all the chiral operations on $j_*j^*\mathcal{A}^{\boxtimes I}$. Then we have
$$
\Delta^{(J/I)*}\mathcal{F}_{X^J}\simeq \mathcal{F}_{X^I}
$$
and $\mathcal{F}$ is a factorization algebra. See \cite[Section 3.4]{beilinson2004chiral} for more details.
\section{Chiral quadratic duality}\label{ChiralQuadratic}
In this section, we will review the notion of the free chiral (factorization) algebra introduced by Beilinson and Drinfeld. We will see that this notion can be viewed as quadratic algebra in the chiral world. We then formulate the chiral quadratic duality and extend it to non-homogeneous cases.

Throughout this section, $X$ denotes a smooth complex algebraic curve and $j:U\hookrightarrow X\times X$ denotes the complement of the diagonal.

\subsection{Quadratic constructions}
We first recall the construction in \cite[Section 3.4.14,pp184]{beilinson2004chiral}.

\begin{defn}
A \textit{chiral quadratic datum} is a pair $(N,P)$ where $N$ is a locally free $\mathbb{Z}$-graded $\mathcal{O}_X$-module of finite rank and $P\subset j_*j^*N\boxtimes N$ is a locally free $\mathcal{O}_{X\times X}$-submodule such that $P|_{U}=N\boxtimes N|_U.$
\end{defn}

\begin{rem}
In the original construction \cite{beilinson2004chiral}, $N$ can be any quasi-coherent $\mathcal{O}_X$-module and $P$ can be any quasi-coherent submodule of $j_*j^*N\boxtimes N$. Here for simplicity, we will only consider the case when both $N$ and $P$ are locally free.
\end{rem}

\begin{rem}
The condition $P|_U=N\boxtimes N|_U$ corresponds to the locality axiom in the definition of vertex algebras. This simply means that for every local section $a\boxtimes b\in N\boxtimes N$, we can find an  integer $n>>0$ sufficiently large such that $(z_1-z_2)^na\boxtimes b\in P$ where $z_1,z_2$ are local coordinates on $X^2$. It translates to the locality axiom: for every pair of generators $(a,b)$ of a vertex algebra, we have
$$
a_{(n)}b=0,\quad n>>0,
$$
here $-_{(n)}-$ is the standard $n$-th product notation in vertex algebras.
\end{rem}
Suppose we have a chiral quadratic datum $(N,P)$. Consider a functor on the category of chiral algebras $\mathcal{CA}(X)$ which assigns to a chiral algebra $\mathcal{A}$ the set of all $\mathcal{O}_X$-linear morphisms
$$\phi:N_{\omega}= N\otimes_{\mathcal{O}_{X}} \omega_X \rightarrow \mathcal{A}$$ such that the chiral product $\mu_{\mathcal{A}}$ annihilates the submodule  $\phi^{\boxtimes 2}(P\otimes _{\mathcal{O}_{X^2}}\omega_{X^2})\subset j_*j^*\mathcal{A}^{\boxtimes 2}$. We denote this functor by $F: \mathcal{CA}(X)\rightarrow \mathcal{S}et$, where $\mathcal{S}et$ is the category of sets.

Beilinson and Drinfeld prove the following theorem.

\begin{thm}
This functor $F$ is representable.
\end{thm}

In \cite{beilinson2004chiral}, they refer to the corresponding universal chiral algebra as the \textit{chiral algebra freely generated by} $(N,P)$. We will denote this chiral algebra by $\mathcal{A}(N,P)$.

\iffalse

\begin{defn}
A chiral algebra $\mathcal{A}$ is called \textit{quadratic} if there is a isomorphism of chiral algebras
$$
\mathcal{A}\simeq \mathcal{A}(N,P)
$$
for a chiral quadratic datum $(N,P)$. In this case, we also say that the chiral algebra $\mathcal{A}$ is represented by the chiral quadratic datum $(N,P)$.
\end{defn}

\begin{rem}
A quadratic chiral algebra $\mathcal{A}$ can be represented by different chiral quadratic datum. We can have two different $(N,P), (N',P')$ such that
$$
\mathcal{A}(N,P)\simeq \mathcal{A}\simeq \mathcal{A}(N',P').
$$
\end{rem}
\fi
\begin{defn}
The \textit{quadratic chiral algebra} associated to a chiral  quadratic datum $(N,P)$ is defined to be $\mathcal{A}(N,P)$.
\end{defn}

The construction of $\mathcal{A}(N,P)$ for general $(N,P)$ is given in \cite[3.4.14,pp184]{beilinson2004chiral}. Fortunately, the details of this construction are not important for this paper as we will only use some formal properties of $\mathcal{A}(N,P)$.

Motivated by the construction of the quadratic duality for quadratic associated algebra, we introduce the quadratic dual relation $P^{\perp}$ as follows.

\begin{defn}
Let $(N,P)$ be a chiral quadratic datum. Define a $\mathcal{O}_{X\times X}$-submodule $P^{\perp}$ of $j_*j^*s^{-1}N^{\vee}_{\omega^{-1}}\boxtimes s^{-1}N^{\vee}_{\omega^{-1}}$ as follows. Consider the following sequence of maps
$$
j_*j^*s^{-1}N^{\vee}_{\omega^{-1}}\boxtimes s^{-1}N^{\vee}_{\omega^{-1}}\xrightarrow{\langle-,-|_P\rangle}\mathrm{Hom}_{\mathcal{O}_{X^2}}(P,j_*j^*s^{-2}\omega_{X^2}^{-1})\rightarrow \mathrm{Hom}_{\mathcal{O}_{X^2}}(P,\frac{j_*j^*s^{-2}\omega_{X^2}^{-1}}{s^{-2}\omega_{X^2}^{-1}}),
$$
where the first map is given by the restriction of the  natural pairing
$$
\langle-,-\rangle:(j_*j^*N\boxtimes N)\otimes_{\mathcal{O}_{X^2}}  (j_*j^*s^{-1}N^{\vee}_{\omega^{-1}}\boxtimes s^{-1}N^{\vee}_{\omega^{-1}})\rightarrow j_*j^*s^{-2}\omega_{X^2}^{-1}
$$
to $P$ and the second map is induced by the quotient map. Let $P^{\perp}$ be the kernel of the composition. In other words, we have
$$
P^{\perp}|_V=\{t|\forall p\in P|_{V},\ \langle t,p\rangle\in s^{-2}\omega_{X^2}^{-1}|_V \}
$$
for any open subset $V$ of $X\times X.$

\iffalse
$$
P^{\perp}:=\mathrm{Hom}(j_*j^*N\boxtimes N,s^{-2}\mathcal{O}_{X\times X})\subset \mathrm{Hom}(j_*j^*N\boxtimes N,s^{-2}j_*j^*\mathcal{O}_{X\times X})= j_*j^*s^{-1}N^{\vee}\boxtimes s^{-1}N^{\vee}.
$$

$$
P^{\perp}:=s^{-2}P^{\vee}\otimes\omega_{X^2}^{-1}=s^{-2}\mathrm{Hom}_{\mathcal{O}_{X^2}}(P,\mathcal{O}_{X^2})\otimes \omega_{X^2}^{-1}.
$$
\fi
\iffalse
Note that $P^{\perp}$ can be viewed as subsheaf of $ j_*j^*s^{-1}N^{\vee}_{\omega^{-1}}\boxtimes s^{-1}N^{\vee}_{\omega^{-1}}.$
\fi

\end{defn}

In general the pair $(s^{-1}N^{\vee}_{\omega^{-1}},P^{\perp})$ is not a chiral quadratic datum. For example, we can take $P$ to be $j_*j^*N\boxtimes N$ itself. Then $P^{\perp}$ will be the zero sheaf and does not satisfy the condition $P|_U=N\boxtimes N|_U$. This leads to the following definition.

\begin{defn}
A chiral quadratic datum $(N,P)$ is called \textit{dualizable} if
$$
P^{\perp}|_{U}=s^{-1}N^{\vee}_{\omega^{-1}}\boxtimes s^{-1}N^{\vee}_{\omega^{-1}}|_U.
$$
%%i.e.  the pair $(s^{-1}N^{\vee}_{\omega^{-1}},P^{\perp})$ is a chiral quadratic datum.

%%A chiral algebra $\mathcal{A}$ is called \textit{dualizable} if it can be represented by a dualizable chiral quadratic datum.
\end{defn}

From the following proposition, we can obtain a dual chiral quadratic datum from a dualizable chiral quadratic datum.

\begin{prop}
If a chiral quadratic datum $(N,P)$ is dualizable, then
$$
P^{\perp}\simeq s^{-2}P^{\vee}\otimes_{\mathcal{O}_{X^2}} \omega^{-1}_{X^2},
$$
here $P^{\vee}=\mathrm{Hom}_{\mathcal{O}_{X^2}}(P,\mathcal{O}_{X^2})$ is the dual of $P$. This implies that $P^{\perp}$ is also locally free and $(N^{\vee}_{\omega^{-1}},P^{\perp})$ is a dualizable chiral quadratic datum.
\end{prop}
\begin{proof}
Assume that $\mathrm{rank}(N)=r$. We first show that there exists a positive integer $k>0$ such that
$$
P\subset N\boxtimes N(k\Delta).
$$

We prove this by contradiction. Suppose that for any positive integer $k>0$, $P$ is not contained in $N\boxtimes N(k\Delta) $. Then we can find an open subset $V\subset
X$ such that $N$ and $N^{\vee}_{\omega^{-1}}$ can be trivialized on $V$ (and we denote a basis of $N|_V$ by $\{e_i\}_{1=1,\dots,r}$) and a sequence of sections
$$
\{\sum_{1\leq i,j\leq r}f^n_{ij}e_i\boxtimes e_j\}_{n\geq 1}, \quad \sum_{1\leq i,j\leq r}f^n_{ij}e_i\boxtimes e_j\in \Gamma(V\times V, P|_{V\times V})
$$
which satisfies that
$$
\{\mathrm{ord}_{\Delta}(f^n_{ij})\}^{1\leq i,j\leq r}_{n\geq 1}
$$
is unbounded below. Here the notation $\mathrm{ord}_{\Delta}$ means the pole order along the diagonal. This means that we can find $(i_0,j_0)\in \{1,\dots,r\}\times \{1,\dots r\}$ and $n_1<n_2<n_3<\cdots$ such that $\{\mathrm{ord}_{\Delta}(f^{n_i}_{i_0j_0})\}_{i\geq 1}$ is unbounded below.  Then we conclude that for $k\in \mathbb{Z}$, $\frac{e^{\vee}_{i_0} dz_1^{-1}\boxtimes e^{\vee}_{j_0}dz_2^{-1}}{(z_1-z_2)^k}\notin \Gamma(V\times V,P^{\perp}|_{V\times V})$. This implies that $P$ is not
dualizable, we get a contradiction.

We conclude that $P$ is a locally free sheaf of rank $r^2$. Then the obvious map
$$
P^{\perp}\rightarrow s^{-2}P^{\vee}\otimes_{\mathcal{O}_{X^2}} \omega^{-1}_{X^2}
$$
is an isomorphism. In fact, we can construct an inverse as follows. We work locally as above, suppose $\{e_i\boxtimes e_j\}$ (resp. $\{p_k\}$) is a local basis of $N\boxtimes N$ (resp. $P$). We can find local functions $\{f^k_{ij}\}, \{f^{-1\ k}_{ij}\}$ regular away from the diagonal such that
$$
p_k=\sum_{1\leq i,j\leq  r} f^k_{ij}e_i\boxtimes e_j,\quad e_i\boxtimes e_j=\sum_{k=1}^{r^2}  f^{-1\ k}_{ij}p_k.
$$
Define
$$
p^{\vee}_k\mapsto \sum_{1\leq i,j\leq  r}  f^{-1\ k}_{ij}e_i^{\vee}\boxtimes e_j^{\vee}.
$$
This defines the desired inverse $s^{-2}P^{\vee}\otimes_{\mathcal{O}_{X^2}}\omega^{-1}_{X^2}\rightarrow P^{\perp}$.
\end{proof}

Now we are ready to introduce the notion of quadratic dual chiral algebra.

\begin{defn}
	\label{Def_quad_dual}
Let $\mathcal{A}(N,P)$ be a  quadratic chiral algebra associated to a dualizable quadratic datum $(N,P)$. We define $\mathcal{A}^!$ to be $$\mathcal{A}(s^{-1}N^{\vee}_{\omega^{-1}},P^{\perp}).$$ We call $\mathcal{A}^!$ the \textit{quadratic dual chiral algebra } of $\mathcal{A}$.
\end{defn}

Since $P$ is locally free, we have $P^{\vee\vee}=P$ which implies that $(\mathcal{A}^{!})^!=\mathcal{A}$. This explains the name of "quadratic dual chiral algebra".

\subsection{Non-homogeneous constructions }\label{Nonhomogeneous1}

In this subsection, we modify the construction in the previous discussion to study the non-homogeneous cases. Namely, we introduce a duality notion that can be viewed as a chiral analogue of non-homogeneous quadratic duality for associative algebras \cite{positsel1993nonhomogeneous}.

Let $\mathbf{1}^{\circ}\simeq \mathcal{O}_X$ be a copy of the trivial line bundle.

\iffalse
We consider an extension of $N$
$$
0\rightarrow \mathbf{1}^{\circ}\rightarrow N^{\flat}\rightarrow N\rightarrow 0.
$$
\fi
\iffalse
\begin{defn}

A  chiral algebra $\mathcal{A}$ is called \textit{quadratic-linear-scalar} (QLS) if  there is a isomorphism of chiral algebras
$$
\mathcal{A}\simeq \frac{\mathcal{A}(N\oplus\mathbf{1}^{\circ},P^{\circ})}{\langle \mathbf{1}^{\circ}-1\rangle}
$$
where $(N\oplus\mathbf{1}^{\circ},P^{\circ})$ is a chiral QLS datum and $\langle \mathbf{1}^{\circ}-1\rangle$ is the ideal generated by $\mathbf{1}^{\circ}-1.$
\end{defn}
\fi

\begin{defn}
	A chiral \textit{quadratic-linear-scalar} (QLS) datum is a chiral quadratic datum in the form of  $(N\oplus \mathbf{1}^{\circ},P^{\circ})$, such that
	$$
	j_*j^*(N\boxtimes \mathbf{1}^{\circ} \oplus \mathbf{1}^{\circ}\boxtimes N \oplus \mathbf{1}^{\circ}\boxtimes \mathbf{1}^{\circ})\cap P^{\circ} = N\boxtimes \mathbf{1}^{\circ} \oplus \mathbf{1}^{\circ}\boxtimes N \oplus \mathbf{1}^{\circ}\boxtimes \mathbf{1}^{\circ}.
	$$
	The QLS chiral algebra associated to a QLS datum $(N\oplus \mathbf{1}^{\circ},P^{\circ})$ is defined to be
		$$
	\mathcal{A}(N,P^{\circ})_{\mathrm{QLS}}:= \frac{\mathcal{A}(N\oplus\mathbf{1}^{\circ},P^{\circ})}{\langle \mathbf{1}^{\circ}_{\omega}-1_{\omega}\rangle}
	$$
	where $1_{\omega}=\omega_X$ is the unit and $\langle \mathbf{1}^{\circ}_{\omega}-1_{\omega}\rangle$ is the ideal generated by $\mathbf{1}^{\circ}_{\omega}-1_{\omega}$.
\end{defn}

For a chiral quadratic datum $(N\oplus\mathbf{1}^{\circ},P^{\circ})$, we denote $\mathbf{q}P^{\circ}\subset j_*j^*N\boxtimes N$ to be the image of
$$
P^{\circ}\hookrightarrow j_*j^*(N\oplus\mathbf{1}^{\circ})\boxtimes (N\oplus\mathbf{1}^{\circ})\rightarrow j_*j^*N\boxtimes N,
$$
where the first arrow is the inclusion and the second arrow is the projection. Using the fact that $(\mathbf{q}P^{\circ})^{\perp}\subset P^{\circ \perp}$, we have the following lemma.
\iffalse
\begin{lem}
The identity map $\mathbf{id}:N\rightarrow N$ induces a morphism of chiral algebras
$$
i:\mathcal{A}(N,\mathbf{q}P^{\circ})\rightarrow  \frac{\mathcal{A}(N\oplus\mathbf{1}^{\circ},P^{\circ})}{\langle \mathbf{1}^{\circ} \rangle}.
$$
\end{lem}
\fi

\begin{lem}
Assume that the chiral quadratic datum $(N\oplus\mathbf{1}^{\circ},P^{\circ})$ is dualizable. Then the identity map $\mathbf{id}:s^{-1}N_{\omega^{-1}}^{\vee}\rightarrow s^{-1}N_{\omega^{-1}}^{\vee}$ induces a injective morphism of chiral algebras
$$
i:\mathcal{A}(s^{-1}N^{\vee}_{\omega^{-1}},(\mathbf{q}P^{\circ})^{\perp})\rightarrow \mathcal{A}(s^{-1}N^{\vee}_{\omega^{-1}}\oplus s^{-1}\mathbf{1}^{\circ}_{\omega^{-1}},P^{\circ\perp}).
$$
\end{lem}
Retain the same notations, we introduce the notion of dualizable chiral QLS datum.

\begin{defn}
We call a chiral QLS datum $(N\oplus \mathbf{1}^{\circ},P^{\circ})$ \textit{dualizable} if $(N\oplus \mathbf{1}^{\circ},P^{\circ})$ is dualizable as a chiral quadratic datum and

1) The inner derivation
$$
d:=\mu(\underline{s^{-1}\mathbf{1}^{\circ}}\boxtimes -): \mathcal{A}(s^{-1}N^{\vee}_{\omega^{-1}}\oplus s^{-1}\mathbf{1}^{\circ}_{\omega^{-1}},P^{\circ\perp})\rightarrow \Delta_*\mathcal{A}(s^{-1}N^{\vee}_{\omega^{-1}}\oplus s^{-1}\mathbf{1}^{\circ}_{\omega^{-1}},P^{\circ\perp})
$$
preserves $\mathrm{Im}(i)$. More precisely, $d(a)$ is in the image of $\Delta_*i$ if $a$ is in the image of $i$;

2) The element $\mu(\underline{s^{-1}\mathbf{1}^{\circ}}\boxtimes \underline{s^{-1}\mathbf{1}^{\circ}})\in \Delta_*\mathcal{A}(s^{-1}N^{\vee}_{\omega^{-1}}\oplus s^{-1}\mathbf{1}^{\circ}_{\omega^{-1}},P^{\circ\perp})$ is in the image of $\Delta_*i$.
\end{defn}
Here, the notation $\underline{s^{-1}\mathbf{1}^{\circ}}$ means the identity global section $\underline{s^{-1}\mathbf{1}^{\circ}} \in \Gamma(X,s^{-1}\mathcal{O}_X)$.

We introduce the notion of twisted pair which later will serve as the "dual chiral algebra" of a QLS chiral algebra.
\begin{defn}
A twisted pair is a triple $(\mathcal{B},\mathcal{B}^{\circ},\mathbf{S})$, where $\mathcal{B}^{\circ}$ is a graded chiral algebra and $\mathcal{B}\subset \mathcal{B}^{\circ}$ is a subalgebra. And $\mathbf{S}\in \Gamma(X,\mathcal{B}^{\circ})$ is a global section of degree -1 such that

1) the map $(h\boxtimes \mathrm{id})\mu(\mathbf{S}\boxtimes -):\mathcal{B}^{\circ}\rightarrow  \mathcal{B}^{\circ}$ preserves the subalgebra $\mathcal{B}.$ Here $h(M):=M\otimes_{\mathcal{D}_X}\mathcal{O}_X$ denotes the de Rham sheaf for any right $\mathcal{D}_X$-module $M$,

2) the element $\mu(\mathbf{S}\boxtimes \mathbf{S})$  belongs to $\Delta_*\mathcal{B}.$

\end{defn}
\iffalse
\begin{rem}
The notion of the twisted pair can be used to define the chiral analogue of the curved Hochshcild chain complex.
\end{rem}
\fi
\iffalse
\begin{prop}
Let $(\mathcal{B},\mathcal{B}^{\circ},\mathbf{S})$ be a twisted pair. Define $$
d:=(h\boxtimes \mathrm{id})\mu(\mathbf{S}\boxtimes -): \mathcal{B}\rightarrow  \mathcal{B},
$$
$$
\iota:=(h\boxtimes \mathrm{id})\mu(\mathbf{S}\boxtimes\mathbf{S})\in \Gamma(X,\mathcal{B}),
$$
then $(\mathcal{B},d,\iota)$ is a CDG chiral algebra.
\end{prop}
\begin{proof}

\end{proof}
\fi

The following proposition is just a reformulation of previous definitions.
\begin{prop}
Let $(N\oplus \mathbf{1}^{\circ},P^{\circ})$ be a dualizable chiral QLS datum.  Then the  triple $$(\mathcal{A}(s^{-1}N^{\vee}_{\omega^{-1}},(\mathbf{q}P^{\circ})^{\perp}),\mathcal{A}(s^{-1}N^{\vee}_{\omega^{-1}}\oplus s^{-1}\mathbf{1}^{\circ}_{\omega^{-1}},P^{\circ\perp}),\underline{s^{-1}\mathbf{1}^{\circ}})$$ is a twisted pair.
\end{prop}

We define the  quadratic dual of the chiral QLS algebra $\mathcal{A}(N,P^{\circ})_{\mathrm{QLS}}$ to be the above twisted pair.

We now introduce the notion of chiral CDG-algebra (curved DG-algebra) which will appear in Section \ref{Examples}.

\begin{defn}
A chiral CDG-algebra is a triple $(\mathcal{B},d,\iota)$, where $\mathcal{B}$ is a graded chiral algebra, $d:\mathcal{B}\rightarrow \mathcal{B}$ is a derivation of $\mathcal{B}$ of degree $-1$, that is, $d$ satisfies
$$
d(\mu(a\boxtimes b))=\mu(da\boxtimes b)+(-1)^{|a|}\mu(a\boxtimes db).
$$
 And $\iota\in \Gamma(X,\mathcal{B})$ is a global section of degree -2 which is called \textit{curving}. It satisfies the following

 1) $d^2(-)=(h\boxtimes \mathrm{id})\mu(\iota\boxtimes -)$,

 2) $d(\iota)=0$.

\end{defn}

We can obtain a chiral CDG-algebra from a twisted pair.
\begin{prop}
Let $(\mathcal{B},\mathcal{B}^{\circ},\mathbf{S})$ be a twisted pair. Define $$
d:=(h\boxtimes \mathrm{id})\mu(\mathbf{S}\boxtimes -): \mathcal{B}\rightarrow  \mathcal{B},
$$
$$
\iota:=(h\boxtimes \mathrm{id})\mu(\mathbf{S}\boxtimes\mathbf{S})\in \Gamma(X,\mathcal{B}),
$$
then $(\mathcal{B},d,\iota)$ is a CDG chiral algebra.
\end{prop}
\begin{proof}
It follows directly from the definition.
\end{proof}
\begin{rem}
In the case of associated algebra, Positselski \cite{positsel1993nonhomogeneous} defines the dual of a QLS algebra to be a CDG algebra constructed from the QLS data. However, in the context of chiral algebras, passing from the twisted pair to the CDG-algebra loses information. Also, the twisted pair is more suitable to construct the curved chiral chain complex which serves as the chiral analogue of the curved Hochschild chain complex in \cite{positsel1993nonhomogeneous}.
\end{rem}
\iffalse
\section{Chiral chain complex of twisted pairs}

We define an injective map of vector spaces (not as complexes)
\iffalse
$$
\mathbf{I}:C^{\mathrm{ch}}(X,\mathcal{A}(s^{-1}N^{\vee},(\mathbf{q}P^{\circ})^{\perp}))\rightarrow C^{\mathrm{ch}}(X,\mathcal{A}(s^{-1}N^{\vee}\oplus \mathbf{1}^{\circ},P^{\circ}))
$$
\fi
$$
\mathbf{I}:C^{\mathrm{ch}}(X,\mathcal{B})\rightarrow C^{\mathrm{ch}}(X,\mathcal{B}^{\circ})
$$
$$
\mathbf{I}(\alpha)=\alpha\boxtimes_{\mathrm{sh}} e^{\boxtimes \mathbf{S}[-1]},
$$
here $\boxtimes_{\mathrm{sh}}$ is the shuffle product.

\begin{defn}
The \textit{twisted chiral chain complex} $C^{\mathrm{ch}}_{\mathbf{tw}}(X,(\mathcal{B},\mathcal{B}^{\circ},\mathbf{S}))$ of a twisted pair $(\mathcal{B},\mathcal{B}^{\circ},\mathbf{S})$ is defined to be
$$
C^{\mathrm{ch}}_{\mathbf{tw}}(X,(\mathcal{B},\mathcal{B}^{\circ},\mathbf{S})):=(C^{\mathrm{ch}}(X,\mathcal{B}),\mathbf{I}^*d^{\circ})
$$
\end{defn}
\fi
\section{Maurer-Cartan equation and quadratic duality}\label{MCEquation}
In this section, we study the relationship between chiral quadratic duality and the Maurer-Cartan equations. In the associative algebra case, it is well known that if an algebra $A$ is Koszul, then the space $\mathrm{MC}(A\otimes B):=\{\alpha\in A\otimes B|[\alpha,\alpha]=0,|\alpha|=-1\}$ of solutions of the Maurer-Cartan equation has a one-to-one correspondence with the space $\mathrm{Hom}(A^!,B)$ of algebra homomorphisms. We study similar correspondence for chiral algebras. However, it is not clear to us how to define the Koszulness for chiral algebras at this stage. Nevertheless, we establish the chiral analogue of this connection for some special cases.

We first introduce the Maurer-Cartan equation for chiral algebras.
\begin{defn}\label{MCeq}
Let $\mathcal{A}$ be a graded chiral algebra.  The Maurer-Cartan equation is defined to be
$$
\mu(\alpha\boxtimes \alpha)=0,\quad \alpha\in \Gamma(X,\mathcal{A}),\ |\alpha|=-1.
$$
The set of the solutions is denoted by $\mathrm{MC}(\mathcal{A}).$
\end{defn}
\begin{rem}
	\label{MC_week}
  Sometimes, one encounters a weaker form of the Maurer-Cartan equation. It has the form of $h(\mu(\alpha\boxtimes \alpha))=0$, where $h(-)=-\otimes_{\mathcal{D}_X}\mathcal{O}_X$ is the de Rham sheaf. For example, \cite{Li:2016gcb} established a correspondence between renormalized quantum master equations and this form of Maurer Cartan equations of vertex algebras.  In the Language of vertex algebras (suppose that $X=\mathbb{C}$), a constant section $vdz$ satisfies the equation in Definition \ref{MCeq} is equivalent to $v_{(n)}v=0$ for $n\geq 0$. While the latter equation is equivalent to $v_{(0)}v=0.$
\end{rem}
We recall the definition of tensor products of chiral algebras. Suppose that $\mathcal{A}_1$ and $\mathcal{A}_2$ are chiral algebras. We denote the corresponding factorization algebras by $\mathcal{F}(\mathcal{A}_i),i=1,2$. Then
$$
\mathcal{F}_{X^I}:=\mathcal{F}(\mathcal{A}_1)_{X^I}\otimes_{\mathcal{O}_{X^I}}\mathcal{F}(\mathcal{A}_2)_{X^I}
$$
is also a factorization algebra. The tensor product $\mathcal{A}_1\otimes \mathcal{A}_2$ is defined to be the chiral algebra that corresponds to $\mathcal{F}$.

\begin{rem}
Suppose $X$ is the complex plane and $\mathcal{A}_i=X\times (V_i)_{\omega}, i=1,2$, where $V_i$ are vertex algebras. Then the above tensor  product is the same as the usual vertex algebra tensor product.
\end{rem}
\iffalse
\begin{defn}
Let $\mathcal{A}$ and $\mathcal{B}$ be two graded chiral algebras.  The Maurer-Cartan equation is define to be
$$
\mu(\alpha\boxtimes \alpha)=0,\quad \alpha\in \mathcal{A}\otimes \mathcal{B},\ |\alpha|=1.
$$
\end{defn}
\fi

\iffalse
From now on we assume that $(N,P)$ is a dualizable effective chiral quadratic datum. By abuse of notation, we write $P\otimes\omega_{X^2}$ instead of $\phi(P\otimes\omega_{X^2})$.
\fi

In the context of quadratic associative algebras, the tensor product of a quadratic algebra and its dual contains a canonical element that satisfies the usual Maurer-Cartan equation. Here we have the chiral algebra version of this.

\begin{prop}\label{MaruerCartan1}
If we take $\mathcal{A}$ to be $\mathcal{A}(N,P)$ and $\mathcal{A}^!$ to be $\mathcal{A}(s^{-1}N^{\vee}_{\omega^{-1}},P^{\perp})$
then the canonical element $\phi(\underline{s^{-1}\mathbf{Id}})\in \Gamma(X,\phi( s^{-1}N^{\vee}\otimes_{\mathcal{O}_{X}} N))\subset \Gamma(X,\mathcal{A}^!\otimes\mathcal{A})$ is a solution to the Maurer-Cartan equation. Here $\phi: s^{-1}N^{\vee}\otimes_{\mathcal{O}_{X}} N\rightarrow \mathcal{A}^!\otimes \mathcal{A}$ is the natural map.

\end{prop}

\begin{proof}
Suppose that $\mathrm{rank}(N)=r$. To simplify the notation, we omit the symbol $\phi$ and pretend that $s^{-1}N^{\vee}\otimes_{\mathcal{O}_{X}} N$ is a submodule of $\mathcal{A}^!\otimes\mathcal{A}$. We can cover $X\times X$ by open subsets, such that we can find a collection of sections $$\{P_{\alpha}\}_{\alpha=1,\dots,r^2}, P_{\alpha}\in P|_{V},$$
and
$$\{P^{\vee}_{\alpha}\}_{\alpha=1,\dots,r^2}, P^{\vee}_{\alpha}\in P^{\perp}\otimes_{\mathcal{O}_{X^2}} \omega_{X^2}|_{{V}},$$
such that
\iffalse
$$
\underline{s^{-1}\mathbf{Id}}\boxtimes \underline{s^{-1}\mathbf{Id}}|_{{V}}=\sum_{\alpha\in S}P_{\alpha}\otimes P^{\vee}_{\alpha}.
$$
\fi
\begin{equation}\label{Identity}
    \underline{s^{-1}\mathbf{Id}}\boxtimes \underline{s^{-1}\mathbf{Id}}|_{{V}}=\sum_{\alpha=1}^{r^2} P^{\vee}_{\alpha}\otimes P_{\alpha}
\end{equation}
for each open subset $V$ that belongs to the covering. By the definition of the tensor product of chiral algebras, we have
$$
 P^{\perp}\otimes_{\mathcal{O}_{X^2}} P\otimes_{\mathcal{O}_{X^2}}\omega_{{X^2}}\subset \ker \mu_{\mathcal{A}^!\otimes \mathcal{A}}.
$$

This implies that $\mu(s^{-1}\mathbf{Id}\boxtimes s^{-1}\mathbf{Id})|_{{V}}=0$ for every $V$. Therefore $\mu(s^{-1}\mathbf{Id}\boxtimes s^{-1}\mathbf{Id})=0$.

\end{proof}

\iffalse
\begin{rem}
Note that the proof is still correct if we substitute $\mathbf{T}(N,P)_{X^2}$ for $P$.
\end{rem}
\fi

Parallel to the quadratic associative algebra case, we can characterize morphisms from a quadratic chiral algebra $\mathcal{A}=\mathcal{A}(N,P)$ to an arbitrary graded chiral algebra $\mathcal{B}$ as solutions of the Maurer-Cartan equations for $\mathcal{A}^!\otimes\mathcal{B}$, i.e., the tensor product of the chiral quadratic dual and the target chiral algebra.

\begin{thm}\label{Representable1}
Let $\mathcal{B}$ be a graded chiral algebra. There exists an injective map
$$
\mathrm{Hom}(\mathcal{A}(N,P),\mathcal{B})\hookrightarrow \mathrm{MC}( \mathcal{A}(s^{-1}N^{\vee}_{\omega^{-1}},P^{\perp})\otimes \mathcal{B}).
$$
\end{thm}
\iffalse
To prove this theorem, we need the following lemma which follows from the construction.

\begin{lem}\label{GeneratorKernel}
If we view $N_{\omega}$ as a submodule of  $ \mathcal{A}(N,P)$. Then
\iffalse
$$
\mathrm{Im}(j_*j^*(i\boxtimes i))\cap \mathrm{ker}(\mu)=\mathbf{T}(N,P)_{X^2}.
$$
\fi
$$
j_*j^*N_{\omega}\boxtimes N_{\omega}\cap \mathrm{ker}(\mu)=\mathbf{T}(N,P)_{X^2}.
$$
\end{lem}
\fi
\begin{proof}
\iffalse
From the previous remark, we can assume that $P^{\perp}=\mathbf{T}(s^{-1}N^{\vee}_{\omega^{-1}},P^{\perp})_{X^2}$. From the above lemma, we have

$$
j_*j^*s^{-1}N^{\vee}\boxtimes N^{\vee}\cap \mathrm{ker}(\mu_{\mathcal{B}})=P^{\perp}\otimes\omega_{X^2}=\mathbf{T}(s^{-1}N^{\vee}_{\omega^{-1}},P^{\perp})_{X^2}\otimes\omega_{X^2}
$$

\fi
Suppose that we have a morphism $\varphi:\mathcal{A}(N,P)\rightarrow \mathcal{B}$. We claim that the element $$(\mathrm{id}\otimes\varphi)(\underline{s^{-1}\mathbf{Id}})\in \Gamma( X,\mathcal{A}(s^{-1}N^{\vee}_{\omega^{-1}},P^{\perp})\otimes \mathcal{B})$$ is a solution of the Maurer-Cartan equation. This claim follows from Proposition \ref{MaruerCartan1} and the fact that $\mathrm{id}\otimes\varphi:\mathcal{A}^!\otimes \mathcal{A}\rightarrow \mathcal{A}^!\otimes\mathcal{B}$ is a morphism of chiral algebras. The injectivity follows from the construction.
\end{proof}

We can show that the above injective map is bijective if we put more conditions. We introduce the notion of effective chiral quadratic datum.

\begin{defn}
A chiral quadratic datum $(N,P)$ is called \textit{effective} if the natural map $\phi:N_{\omega}\rightarrow \mathcal{A}(N,P)$ is injective and (for simplicity of notation, we will omit the symbol $\phi$)
$$
P\otimes_{\mathcal{O}_X} \omega_{X^2}=j_*j^*N_{\omega}\boxtimes N_{\omega}\cap\ker\mu_{\mathcal{A}(N,P)}.
$$

\end{defn}

\begin{rem}
  It is easy to find effective chiral quadratic datum. We can start from an arbitrary chiral quadratic datum $(N,P)$. If $P'\otimes_{\mathcal{O}_X} \omega_{{X^2}} =j_*j^*N_{\omega}\boxtimes N_{\omega}\cap\ker\mu_{\mathcal{A}(N,P)}$ is locally free, then we can take $(N,P')$ to be our new chiral quadratic datum. From the construction in \cite[3.4.14,pp184]{beilinson2004chiral}, we have $\mathcal{A}(N,P)=\mathcal{A}(N,P')$ and $(N,P')$ is effective.
\end{rem}
\begin{thm}\label{Representable2}
Let $\mathcal{B}$ be a graded chiral algebra which concentrated in degree 0. Assume that $N$ is degree 0 and $(s^{-1}N^{\vee}_{\omega^{-1}},P^{\perp})$ is effective, then there exists a bijection
$$
\mathrm{Hom}(\mathcal{A}(N,P),\mathcal{B})\cong \mathrm{MC}(\mathcal{A}(s^{-1}N^{\vee}_{\omega^{-1}},P^{\perp})\otimes\mathcal{B}).
$$
\end{thm}

\begin{proof}

We omit the symbol $\phi$ as before. We use the notation $\mathcal{A}=\mathcal{A}(N,P), \mathcal{A}^!=\mathcal{A}(s^{-1}N^{\vee}_{\omega^{-1}},P^{\perp})$. Suppose that we have $\alpha\in \mathcal{A}^!\otimes \mathcal{B},|\alpha|=-1$ satisfies the Maurer-Cartan equation. Since we assume that both $\mathcal{B}$ and $N$ are in degree 0, we have
$$
\alpha\in s^{-1}N^{\vee}_{\omega^{-1}}\otimes_{\mathcal{O}_X}\mathcal{B} \subset \mathcal{A}^!\otimes \mathcal{B}.
$$
Then $\alpha$ defines a morphism of $\mathcal{O}_X$ modules
$$
\phi_{\alpha}:N_{\omega}\rightarrow  \mathcal{B},
$$
$$
\phi_{\alpha}(-)=\langle s\alpha,-\rangle.
$$
Note that we have
$$
(\mathrm{id}\otimes\phi_{\alpha})(\underline{s^{-1}\mathbf{Id}})=\alpha.
$$
\iffalse
Repeat the proof in Proposition \ref{MaruerCartan1}, for any $(x,x)\in X\times X$ and open neighbourhood $(x,x)\in V\times V\subset X\times X$,
\fi
We can cover $X^2$ by open subsets $\cup V_i$. We can find $\{P^i_{\alpha}\},\{P^{i\ \vee}_{\alpha}\}$such that the equation \ref{Identity} holds on $V_i$. Now take $V=V_i$, we have
$$
0=\mu(\alpha\boxtimes \alpha)|_{\mathbf{V}}=\mu((\mathrm{id}\boxtimes \mathrm{id})\otimes (\phi_{\alpha}\boxtimes \phi_{\alpha}) (\underline{s^{-1}\mathbf{Id}}\boxtimes \underline{s^{-1}\mathbf{Id}}))|_{{V}}
$$
$$
=\mu((\mathrm{id}\boxtimes \mathrm{id})\otimes (\phi_{\alpha}\boxtimes \phi_{\alpha})(\sum_{\alpha\in S}P^{\vee}_{\alpha}\otimes P_{\alpha}))|_{{V}}
$$
$$
=\mu(\sum_{\alpha\in S} P^{\vee}_{\alpha}\otimes (\phi_{\alpha}\boxtimes \phi_{\alpha})(P_{\alpha}))|_{{V}}.
$$

 We have $$
(\sum_{\alpha\in S} P^{\vee}_{\alpha}\otimes (\phi_{\alpha}\boxtimes \phi_{\alpha})(P_{\alpha}))|_{{V}}=\sum_{\alpha\in S} P^{\vee}_{\alpha}\otimes Q_{\alpha},\quad Q_{\alpha}\in \mathrm{ker}(\mu_{\mathcal{B}})|_{{V}}
$$
since we assume that $(s^{-1}N^{\vee}_{\omega^{-1}},P^{\perp})$ is effective. This implies that
$$
\mu_{\mathcal{B}}((\phi_{\alpha}\boxtimes \phi_{\alpha})(P_{\alpha}) )|_{{V}}=\mu_{\mathcal{B}}(Q_{\alpha})=0,\quad \alpha \in S.
$$
\iffalse
This means that for any point $(x,x)\in X\times X$, we can find an open neighbourhood of that point such that
$$
\mu_{\mathcal{A}}((\phi_{\alpha}\boxtimes \phi_{\alpha})(p) )=0, \quad \forall p\in P
$$
in that open neighbourhood.
\fi
\end{proof}

We can generalize the notion of the Maurer-Cartan equation to twisted pairs.
\begin{defn}\label{MC_twist}
Let $\mathcal{A}$ be a graded chiral algebra and $(\mathcal{B},\mathcal{B}^{\circ},\mathbf{S})$ be a twisted pair.  The Maurer-Cartan equation is define to be
$$
\mu((\mathbf{S}+\alpha)\boxtimes (\mathbf{S}+\alpha))=0,\quad \alpha\in \Gamma(X,\mathcal{A}\otimes \mathcal{B}),\ |\alpha|=-1.
$$
The set of the solutions is denoted by $\mathrm{MC}(\mathcal{(\mathcal{B},\mathcal{B}^{\circ},\mathbf{S})\otimes \mathcal{A}})$
\end{defn}

\begin{prop}
If we take $\mathcal{A}$ to be $\frac{\mathcal{A}(N\oplus\mathbf{1}^{\circ},P^{\circ})}{\langle \mathbf{1}^{\circ}_{\omega}-1_{\omega}\rangle}$ and $(\mathcal{B},\mathcal{B}^{\circ},\mathbf{S})$ to be
\iffalse
$$(\mathcal{A}(s^{-1}N^{\vee},(\mathbf{q}P^{\circ})^{\perp}),\mathcal{A}(s^{-1}N^{\vee}\oplus s^{-1}\mathbf{1}^{\circ},P^{\circ\perp}),\underline{s^{-1}\mathbf{1}^{\circ}})$$
\fi
\begin{equation}\label{TwistedPair}
  (\mathcal{A}(s^{-1}N^{\vee}_{\omega^{-1}},(\mathbf{q}P^{\circ})^{\perp}),\mathcal{A}(s^{-1}N^{\vee}_{\omega^{-1}}\oplus s^{-1}\mathbf{1}^{\circ}_{\omega^{-1}},P^{\circ\perp}),\underline{s^{-1}\mathbf{1}^{\circ}})
\end{equation}
the canonical element $\underline{s^{-1}\mathbf{Id}}\in  \Gamma(X,s^{-1}N^{\vee}\otimes_{\mathcal{O}_X} N)\subset \Gamma(X,\mathcal{B}\otimes \mathcal{A})$ is a solution to the Maurer-Cartan equation.
\end{prop}

\begin{proof}
The identity element $\underline{s^{-1}\mathbf{Id}}^{\circ}\in \Gamma(X,(s^{-1}N^{\vee}_{\omega^{-1}}\oplus s^{-1}\mathbf{1}^{\circ}_{\omega^{-1}})\otimes_{\mathcal{O}_X}  (N\oplus \mathbf{1}^{\circ})\otimes_{\mathcal{O}_X}\omega_X)$ satisfies the usual Maurer-Cartan equation in $ \mathcal{B}^{\circ}\otimes\mathcal{A}(N\oplus\mathbf{1}^{\circ},P^{\circ})$
$$
\mu(\underline{s^{-1}\mathbf{Id}}^{\circ}\boxtimes \underline{s^{-1}\mathbf{Id}}^{\circ})=0.
$$
Note that $\underline{s^{-1}\mathbf{Id}}^{\circ}=\underline{s^{-1}\mathbf{Id}}+\mathbf{S}\in \Gamma(X,\mathcal{B}^{\circ}\otimes \mathcal{A})$, the proposition follows.
\end{proof}
\begin{thm}
Let $\mathcal{C}$ be a graded chiral algebra and $(\mathcal{B},\mathcal{B}^{\circ},\mathbf{S})$ be the twisted pair (\ref{TwistedPair}). Then there is a injection
$$
\mathrm{Hom}(\frac{\mathcal{A}(N\oplus\mathbf{1}^{\circ},P^{\circ})}{\langle \mathbf{1}^{\circ}_{\omega}-1_{\omega}\rangle},\mathcal{C})\hookrightarrow\mathrm{MC}((\mathcal{B},\mathcal{B}^{\circ},\mathbf{S})\otimes \mathcal{C}).
$$
\end{thm}
\begin{proof}
Suppose there is a morphism of chiral algebras
$$
\phi:\frac{\mathcal{A}(N\oplus\mathbf{1}^{\circ},P^{\circ})}{\langle \mathbf{1}^{\circ}_{\omega}-1_{\omega}\rangle}\rightarrow \mathcal{C}.
$$
Note that $\phi$ is induced by the following morphism
$$
\tilde{\phi}:\mathcal{A}(N\oplus\mathbf{1}^{\circ},P^{\circ})\rightarrow \mathcal{C}
$$
such that $\tilde{\phi}|_{N_{\omega}}=\phi|_{N_{\omega}} $ and $\tilde{\phi}|_{\mathbf{1}^{\circ}_{\omega}}=1_{\omega}=\omega_X.$
Then
$$
\alpha=\tilde{\phi}(\underline{s^{-1}\mathbf{Id}}^{\circ})-\mathbf{S}=\phi(\underline{s^{-1}\mathbf{Id}})
$$
is the solution of the Maurer-Cartan equation.

\end{proof}

Similarly, we have the following theorem.

\begin{thm}
Let $\mathcal{C}$ be a graded chiral algebra concentrated in degree 0 and $(\mathcal{B},\mathcal{B}^{\circ},\mathbf{S})$ be the twisted pair (\ref{TwistedPair}). Assume that $N$ is degree 0 and $(s^{-1}N^{\vee}_{\omega^{-1}}\oplus s^{-1}\mathbf{1}^{\circ}_{\omega^{-1}},P^{\circ\perp})$ is effective. Then there is a bijection
$$
\mathrm{Hom}(\frac{\mathcal{A}(N\oplus\mathbf{1}^{\circ},P^{\circ})}{\langle \mathbf{1}^{\circ}_{\omega}-1_{\omega}\rangle},\mathcal{C})\cong\mathrm{MC}((\mathcal{B},\mathcal{B}^{\circ},\mathbf{S})\otimes \mathcal{C}).
$$
\end{thm}
\begin{proof}

Suppose we have a solution $\alpha$ of the Maurer-Cartan equation. We can define a map
$\tilde{\phi}_{\alpha}:N_{\omega}\oplus\mathbf{1}^{\circ}_{\omega}\rightarrow \mathcal{C}$ such that
$$
\alpha=(\mathrm{id}\otimes \tilde{\phi}_{\alpha}|_{N_{\omega}})(\underline{s^{-1}\mathbf{Id}}),
$$
and $\tilde{\phi}_{\alpha}|_{\mathbf{1}^{\circ}_{\omega}}:\mathbf{1}^{\circ}_{\omega}\rightarrow \mathcal{C}$ is equal to the unit map $\omega_X\rightarrow\mathcal{C}$. Then repeat the proof in Theorem \ref{Representable1}, we have a morphism of chiral algebras
$$
\tilde{\phi}_{\alpha}:\mathcal{A}(N\oplus\mathbf{1}^{\circ},P^{\circ})\rightarrow \mathcal{C},
$$
and it factors through the ideal $\langle \mathbf{1}^{\circ}_{\omega}-1_{\omega}\rangle$ by construction. Thus, we have a morphism
$$
\phi_{\alpha}:\frac{\mathcal{A}(N\oplus\mathbf{1}^{\circ},P^{\circ})}{\langle \mathbf{1}^{\circ}_{\omega}-1_{\omega}\rangle}\rightarrow \mathcal{C}.
$$
The proof is complete.

\end{proof}

\section{Examples}\label{Examples}

There are some classical examples of Koszul duality for associative algebra. The most famous examples of Koszul dual algebras are the symmetric algebra $S(V)$ and the exterior algebra $\wedge V^{\vee}$. In the non-homogeneous case, we have the Koszul duality between the universal enveloping algebra $U(\mathfrak{g})$ and the Chevalley-Eilenberg algebra $\mathrm{CE}(\mathfrak{g})$. In this section, we discuss examples of quadratic duality for chiral algebra that parallel the cases of associative algebras.

\subsection{Commutative chiral algebra}
First, we consider the simplest quadratic datum $(N, P = N\boxtimes N)$, with $N$ locally free of finite rank. We have $\mathcal{A}(N,P) = \mathrm{Sym}(N_{\omega\mathcal{D}}) $, which is the commutative chiral algebra generated by $N_{\omega\mathcal{D}} : = N_{\omega}\otimes_{\mathcal{O}_X}\mathcal{D}_X$.

The dual quadratic datum is given by $(s^{-1}N^{\vee}_{\omega^{-1}}, P^{\perp} = s^{-1}N^{\vee}_{\omega^{-1}}\boxtimes s^{-1}N^{\vee}_{\omega^{-1}})$. It automatically satisfies $P^{\perp}|_{U} = s^{-1}N^{\vee}_{\omega^{-1}}\boxtimes s^{-1}N^{\vee}_{\omega^{-1}}|_U$, so this quadratic datum is dualizable. We have $\mathcal{A}(s^{-1}N^{\vee}_{\omega^{-1}},P^{\perp} ) =  \mathrm{Sym}( (s^{-1}N^{\vee})_{\mathcal{D}}) $, which is the graded commutative chiral algebra generated by $(s^{-1}N^{\vee})_{\mathcal{D}}:=s^{-1}N^{\vee}\otimes_{\mathcal{O}_X}\mathcal{D}_X$.

\subsection{Another pure quadratic example}
Let $N$ be the free $\mathcal{O}_X$-module $ N =\bigoplus_{i  = 1}^4\mathcal{O}_X$. We denote the corresponding basis by $\{\phi_i\}_{i = 1,\dots 4}$. We define $P$ to be the $\mathcal{O}_{X^2}$ module with basis
$$
\begin{aligned}
	&\phi_i\boxtimes \phi_j ,\;\{i,j\} \neq \{1,2\},\\
	&\phi_1 \boxtimes \phi_2 - \frac{1}{z_1 - z_2} \phi_3\boxtimes \phi_4,\\
	&\phi_2 \boxtimes \phi_1 + \frac{1}{z_1 - z_2} \phi_4\boxtimes \phi_3.\\
\end{aligned}
$$

For the dual datum, we have $s^{-1}N^{\vee}_{\omega^{-1}} = \bigoplus_{i  = 1}^4s^{-1}\omega^{-1}_X$. We denote the corresponding basis by $\{\psi_i = s^{-1}\phi_i^{\vee}\}_{i = 1,\dots 4}$. Then $P^{\perp}$ has the following basis
$$
\begin{aligned}
	&\psi_i\boxtimes \psi_j ,\;\{i,j\} \neq \{3,4\},\\
	&\psi_3 \boxtimes \psi_4 + \frac{1}{z_1 - z_2} \psi_1\boxtimes \psi_2,\\
	&\psi_4 \boxtimes \psi_3 - \frac{1}{z_1 - z_2} \psi_2\boxtimes \psi_1.
\end{aligned}
$$
$P^{\perp}$ defined above satisfies $P^{\perp}|_U = s^{-1}N^{\vee}_{\omega^{-1}}\boxtimes s^{-1}N^{\vee}_{\omega^{-1}}|_U$, so this quadratic datum is dualizable.

\subsection{Affine Kac-Moody chiral algebra}
Let $\mathfrak{g}$ be a finite dimensional Lie algebra with an invariant pairing $\kappa$. We take a basis $\{x_a\}_{1\leq a \leq n}$ of $\mathfrak{g}$.	Let $N = \mathfrak{g}\otimes \omega^{-1}_X$. We consider $P^{\circ} \subset j_*j^* (N\oplus \mathbf{1}^{\circ})\boxtimes(N\oplus \mathbf{1}^{\circ})$ be the $\mathcal{O}_{X^2}$-module defined by the following basis
\begin{equation}\label{rel_KacMoody}
	\begin{aligned}
		&\mathbf{1}^\circ\boxtimes\mathbf{1}^\circ,\\
		&\mathbf{1}^{\circ}\boxtimes x_a,\quad  x_a\boxtimes \mathbf{1}^{\circ},\quad 1\leq a\leq n,\\
		&	x_a\boxtimes y_b - \frac{1}{2}\sum_{c = 1}^n(\frac{f^c_{ab}}{z_1 - z_2})\left( \mathbf{1}^{\circ}\boxtimes x_c + x_c\boxtimes\mathbf{1}^{\circ}  \right)  - \frac{\kappa_{ab}}{(z_1 - z_2)^2}\mathbf{1}^{\circ}\boxtimes \mathbf{1}^{\circ},; 1\leq a,b\leq n,
	\end{aligned}
\end{equation}
where $\kappa_{ab} = \kappa(x_a,x_b)$.

As a more familiar construction, we consider the affine Kac-Moody Lie$^*$ algebra $\mathfrak{g}_{\mathcal{D}}^{\kappa} = \mathfrak{g}_{\mathcal{D}} \oplus \omega_X$. It gives rise to the twisted chiral enveloping algebra $U(\mathfrak{g}_{\mathcal{D}})^{\kappa}$ \cite[Section 3.7.25, pp227]{beilinson2004chiral}.
\begin{prop}\label{proof_envelop}
	We have an isomorphism of chiral algebra
	$$
	\frac{	\mathcal{A}(N\oplus \mathbf{1}^{\circ},P^\circ)}{\langle \mathbf{1}^{\circ}_{\omega}-1_{\omega}\rangle} = U(\mathfrak{g}_{\mathcal{D}})^{\kappa} .
	$$
\end{prop}

\begin{proof}
	On the one hand, we have a map $N_{\omega}\oplus \mathbf{1}^{\circ}_{\omega} \to \mathfrak{g}_{\mathcal{D}}^{\kappa} \to U(\mathfrak{g}_{\mathcal{D}})^{\kappa}$. By the universal property, we get a map of chiral algebra $	\mathcal{A}(N\oplus \mathbf{1}^{\circ},P) \to U(\mathfrak{g}_{\mathcal{D}})^{\kappa}$. By construction, $\mathbf{1}^{\circ}_{\omega}$ is mapped to the unit of $U(\mathfrak{g}_{\mathcal{D}})^{\kappa}$. Therefore we have a map of chiral algebra $\frac{	\mathcal{A}(N\oplus \mathbf{1}^{\circ},P^\circ)}{\langle \mathbf{1}^{\circ}_{\omega}-1_{\omega}\rangle} \to U(\mathfrak{g}_{\mathcal{D}})^{\kappa}$.
	
	On the other hand, we consider the map $N_{\omega}\oplus \mathbf{1}^{\circ}_{\omega} \to \frac{\mathcal{A}(N\oplus \mathbf{1}^{\circ},P^\circ)}{\langle \mathbf{1}^{\circ}_{\omega}-1_{\omega}\rangle} $, which extends to a $\mathcal{D}_X$-module map $\mathfrak{g}_\mathcal{D}^\kappa \to \frac{	\mathcal{A}(N\oplus \mathbf{1}^{\circ},P^\circ)}{\langle \mathbf{1}^{\circ}_{\omega}-1_{\omega}\rangle}$. Using the relation \ref{rel_KacMoody}, we find that the image of this map has the same Lie$^*$ bracket as $\mathfrak{g}_\mathcal{D}^\kappa$. Therefore we get a map of Lie$^*$ algebra $\mathfrak{g}_\mathcal{D}^\kappa \to \frac{	\mathcal{A}(N\oplus \mathbf{1}^{\circ},P^\circ)}{\langle \mathbf{1}^{\circ}_{\omega}-1_{\omega}\rangle} $. By the universal property of (twisted) chiral envelope, we have a map of chiral algebra $ U(\mathfrak{g}_{\mathcal{D}})^{\kappa} \to \frac{\mathcal{A}(N\oplus \mathbf{1}^{\circ},P^\circ)}{\langle \mathbf{1}^{\circ}_{\omega}-1_{\omega}\rangle} $.
	
	The composition $\mathfrak{g}_\mathcal{D}^\kappa \to \frac{\mathcal{A}(N\oplus \mathbf{1}^{\circ},P^\circ)}{\langle \mathbf{1}^{\circ}_{\omega}-1_{\omega}\rangle} \to U(\mathfrak{g}_{\mathcal{D}})^{\kappa}$ is the canonical map $\mathfrak{g}_\mathcal{D}^\kappa \to U(\mathfrak{g}_{\mathcal{D}})^{\kappa}$. Therefore the composition $U(\mathfrak{g}_{\mathcal{D}})^{\kappa} \to \frac{\mathcal{A}(N\oplus \mathbf{1}^{\circ},P^\circ)}{\langle \mathbf{1}^{\circ}_{\omega}-1_{\omega}\rangle} \to U(\mathfrak{g}_{\mathcal{D}})^{\kappa}$ is the identity. Similarly the composition $\frac{\mathcal{A}(N\oplus \mathbf{1}^{\circ},P^\circ)}{\langle \mathbf{1}^{\circ}_{\omega}-1_{\omega}\rangle} \to U(\mathfrak{g}_{\mathcal{D}})^{\kappa} \to \frac{\mathcal{A}(N\oplus \mathbf{1}^{\circ},P^\circ)}{\langle \mathbf{1}^{\circ}_{\omega}-1_{\omega}\rangle} $ also gives the identity.
\end{proof}

Now we analyze the quadratic dual datum. We find that $P^{\circ\perp}$ is given by the following basis
$$
\begin{aligned}
	&s^{-1}\mathbf{1}^{\circ}_{\omega^{-1}}\boxtimes s^{-1}\mathbf{1}^{\circ}_{\omega^{-1}}+ \sum_{1\leq a,b\leq n}\frac{\kappa_{ab}}{(z_1 - z_2)^2}s^{-1}x_a^\vee \boxtimes s^{-1}x^\vee_b,\\
	&s^{-1}\mathbf{1}^{\circ}_{\omega^{-1}}\boxtimes s^{-1}x_c^\vee + \frac{1}{2}\sum_{1\leq a,b\leq n}\frac{f_{ab}^c}{z_1 - z_2}s^{-1}x_a^\vee \boxtimes s^{-1}x^\vee_b,\quad 1\leq c\leq n,\\
	&s^{-1}x_c^\vee \boxtimes s^{-1}\mathbf{1}^{\circ}_{\omega^{-1}}  + \frac{1}{2}\sum_{1\leq a,b\leq n}\frac{f_{ab}^c}{z_1 - z_2}s^{-1}x_a^\vee \boxtimes s^{-1}x^\vee_b,\quad 1\leq c\leq n,\\
	& s^{-1}x_a^\vee \boxtimes s^{-1}x^\vee_b,\quad 1\leq a,b\leq n.
\end{aligned}
$$

We see that $(N\oplus \mathbf{1}^{\circ},P^\circ)$ is dualizable as qudratic datum. The quadratic projection $(\mathbf{q}P^{\circ})^{\perp}$ is given by the following basis
$$
s^{-1}x_a^\vee \boxtimes s^{-1}x^\vee_b,\quad 1\leq a,b\leq n.
$$
Therefore, the chiral algebra $\mathcal{B} = \mathcal{A}(s^{-1}N^{\vee}_{\omega^{-1}},(\mathbf{q}P^{\circ})^{\perp}) =  \mathrm{Sym}((s^{-1}N^{\vee})_{\mathcal{D}})$ is the graded commutative chiral algebra generated by $s^{-1}N^{\vee}=s^{-1}\mathfrak{g}^{\vee}\otimes_{\mathcal{O}_X}\omega_X$.

We denote $\mathcal{B}^{\circ} = \mathcal{A}(s^{-1}N^{\vee}_{\omega^{-1}}\oplus s^{-1}\mathbf{1}^{\circ}_{\omega^{-1}},P^{\circ\perp})$. To prove that $(\mathcal{B},\mathcal{B}^\circ,\underline{s^{-1}\mathbf{1}^{\circ}})$ is indeed a twisted pair, we analyze the differential and the curving element.
\begin{prop}\label{proof_CE}
	The differential defined by $d =  (h\boxtimes \mathrm{id})\mu(\underline{s^{-1}\mathbf{1}^{\circ}}\boxtimes -)$ preserves $\mathcal{B} $. Moreover, the DG chiral algebra $(\mathcal{B},d)$ is isomorphic to the Chevalley DG algebra $ (\EuScript{C}(\mathfrak{g}_{\mathcal{D}}),d_{\mathrm{CE}})$ for the Lie$^*$ algebra $\mathfrak{g}_{\mathcal{D}}$ (see \cite[Section 4.7, pp348]{beilinson2004chiral} for details, where they use the name "de Rham-Chevalley algebra" as the construction is for general Lie$^*$ algebroids).
\end{prop}

\begin{proof}
	$\mathcal{B}$ is a commutative chiral algebra, which coincide with $\EuScript{C}(\mathfrak{g}_{\mathcal{D}})$ as plain graded chiral algebra. The corresponding left $\mathcal{D}$-module $\mathcal{B}^l$ is a commutative $\mathcal{D}_X$-algebra.
	
	We denote the image of $s^{-1}x_c^\vee$ under $s^{-1}N^{\vee}_{\omega^{-1}}=s^{-1}\mathfrak{g}^{\vee} \to \mathcal{B}^l$ by the same symbol $s^{-1}x_c^\vee$. Using the dual relation we can compute $d$ restricted to $s^{-1}N^{\vee}_{\omega^{-1}} \otimes \omega_X$ as follows
	$$
	\begin{aligned}
		d(s^{-1}x_c^\vee dz) & = (h\boxtimes \mathrm{id})\mu(s^{-1}\mathbf{1}^{\circ} \boxtimes s^{-1}x_c^\vee dz_2)\\
		& =  \frac{1}{2}\sum_{1\leq a,b\leq n} (h\boxtimes \mathrm{id})\mu(\frac{f_{ab}^c}{z_1 - z_2}s^{-1}x_a^\vee dz_1\boxtimes s^{-1}x^\vee_b dz_2).
	\end{aligned}
	$$
	The chiral operation $\mu$ restricted to $\mathcal{B}$ is given by the commutative product on $\mathcal{B}^l$. We can simplify the above map as follows
	$$
	d(s^{-1}x_c^\vee dz) = \frac{1}{2}\sum_{1\leq a,b\leq n} f_{ab}^c(s^{-1}x_a^\vee \cdot s^{-1}x^\vee_b)dz.
	$$
	Since $d$ is a $\mathcal{D}$-module map, the above result extend to a map $d: (s^{-1}N^{\vee})_\mathcal{D} \to \mathcal{B}$. We see that $d$ restricted to $(s^{-1}N^{\vee})_\mathcal{D}$ is given by the composition $(s^{-1}N^{\vee})_\mathcal{D} \overset{[-,-]^*}{\to} (s^{-1}N^{\vee})_\mathcal{D}\otimes(s^{-1}N^{\vee})_\mathcal{D} \to \mathrm{Sym}^2((s^{-1}N^{\vee})_\mathcal{D})$, which coincide with $d_{CE}$.
	
	The Jacobi identity of the chiral operation implies that $d$ satisfies the Leibniz rule. We thus complete the proof.
	
\end{proof}

The final ingredient is the curving. Using the dual relation we find that it is given by
$$
\begin{aligned}
	\iota & = (h\boxtimes \mathrm{id})\mu(s^{-1}\mathbf{1}\boxtimes s^{-1}\mathbf{1})\\
	& =  -\sum_{1\leq a,b\leq n} (h\boxtimes \mathrm{id})\mu(\frac{\kappa_{ab}}{(z_1 - z_2)^2}s^{-1}x_a^\vee dz_1 \boxtimes s^{-1}x^\vee_bdz_2).\\
\end{aligned}
$$
We see that $(h\boxtimes \mathrm{id})\mu(\frac{\kappa_{ab}}{(z_1 - z_2)^2}s^{-1}x_a^\vee dz_1\boxtimes s^{-1}x^\vee_bdz_2)$ is indeed an element of $\mathcal{B}$. Therefore the triple $(\mathcal{B},\mathcal{B}^\circ,\underline{s^{-1}\mathbf{1}^{\circ}})$ is a twisted pair and serves as the quadratic dual of $U(\mathfrak{g}_{\mathcal{D}})^{\kappa}$.

From the vertex algebra point of view, the vertex algebra corresponding to the twisted chiral envelope $U(\mathfrak{g}_{\mathcal{D}})^{\kappa}$ is the affine Kac-Moody VOA $V_{\kappa}(\mathfrak{g})$. The quadratic dual vertex algebra can be identified with the graded commutative vertex algebra $V^{CE}(\mathfrak{g}) := CE(L\mathfrak{g})$ equipped with the Chevalley-Eilenberg differential and a curving. Explicitly, we denote $\{J_a(z) = \sum\limits_{n\in \mathbf{Z} }J_{a,(n)}z^{-n - 1}\}_{1\leq a \leq n}$ the set of generating fields of $V_{\kappa}(\mathfrak{g})$. The quadratic dual vertex algebra $V^{CE}(\mathfrak{g})$ is generated by fields $\{c^{a}(z) = \sum\limits_{n\in\mathbf{Z}}c^{a}_{(n)}z^{-n - 1}\}_{1\leq a \leq n}$. The differential can be expressed as follows
$$
d (\partial^mc^a ) = -\frac{1}{2}\sum_{1\leq b,c \leq n}\sum_{r+s = m}f^{a}_{bc}\binom{m}{r}(\partial^r c^b)(\partial^s c^c),
$$
where we define $\partial^mc^a = \partial^mc^a(0)|0\rangle$. Using VOA axiom, $\partial^mc^a$ can also be identified with $T^mc^a$. The curving element can be identified with
$$
\iota = -\sum_{1\leq a,b \leq n}\kappa_{ab}(\partial c^a )c^b.
$$
The canonical element $s^{-1}\mathbf{Id} \in \Gamma(X,s^{-1}N\otimes_{\mathcal{O}_{X}} N)$ corresponds to the following element in the vertex algebra $V_{\kappa}(\mathfrak{g})\otimes V^{CE}(\mathfrak{g})$
$$
\mathbf{I} := \sum_{a = 1}^n J_a \otimes c^a.
$$

We can verify the corresponding Maurer-Cartan equation using vertex algebra operation. Note that $\mathbf{I}_{(0)} = \sum\limits_{1\leq a \leq n}\sum\limits_{l+m = -1}J_{a,(l)} \otimes c^a_{(m)}$. We find the following
$$
\mathbf{I}_{(0)}\mathbf{I} = \sum_{1\leq a,b ,c\leq n}f_{ab}^c J_c\otimes c^ac^b + \sum_{1\leq a,b\leq n}|0\rangle\otimes\kappa_{ab}(\partial c^a )c^b.
$$
We also have
$$
d\mathbf{I} = -\frac{1}{2}\sum_{1\leq a,b,c\leq n}f^a_{bc}J_a\otimes c^bc^c.
$$
Therefore, the following Maurer-Cartan equation is satisfied
\begin{equation}\label{VOA_MC1}
	d\mathbf{I} + \frac{1}{2}\mathbf{I}_{(0)}\mathbf{I} + \frac{1}{2}\iota = 0.
\end{equation}
We can use $\mathbf{I}_{(m)} = \sum\limits_{1\leq a \leq n}\sum\limits_{l+k = m-1}J_{a,(l)} \otimes c^a_{(k)}$ to check that the stronger form of Maurer-Cartan equation (see the Remark \ref{MC_week}) is also satisfied
\begin{equation}\label{VOA_MC2}
	\mathbf{I}_{(m)}\mathbf{I} = 0 ,\text{ for } m \geq 1.
\end{equation}

As a consequence, for any vertex algebra $V$ and a homomorphism $\varphi:V_{\kappa}(\mathfrak{g}) \to V $, $(\varphi\otimes \mathrm{id})(\mathbf{I})$ satisfies the Maurer-Cartan equation. On the other hand, for any vertex algebra $V$ concentrated in degree $0$, a degree $1$ element of $V^{CE}(\mathfrak{g})\otimes V$ takes the following form
$$
\alpha=\sum_{a = 1}^nc^a\otimes y_a,\; y_a \in V.
$$
For the (strong form of) Maurer-Cartan equation \ref{VOA_MC1},\ref{VOA_MC2} to hold for $\alpha$, we must have
$$
\begin{aligned}
	&y_{a,(0)}y_b = \sum_{c = 0}^n f_{ab}^cy_c,\\
	& y_{a,(1)}y_b = \kappa_{ab}|0\rangle , \\
	&y_{a,(m)}y_b = 0,\text{ for } m\geq 2.
\end{aligned}
$$
Using Borcherds identities, we find
$$
[y_{a,(l)},y_{b,(m)}] = \sum_{c = 1}^nf_{ab}^c y_{a,(l+m)} + \kappa_{ab}\delta_{n,-m}.
$$
This implies that the following map
$$
J_a \to y_a,\quad 1\leq a\leq n,
$$
defined a homomorphism of vertex algebra $V_{\kappa}(\mathfrak{g}) \to V$.
\subsection{$\beta\gamma-bc$ system}
Let $\mathbf{L} = \mathop{\bigoplus}\limits_{\alpha \in \mathbb{Q}}\mathbf{L}^{\alpha}$ be a finite dimensional $\mathbb{Q}$(conformal weight)-graded superspace. Suppose that $\mathbf{L}$ is equipped with an even symplectic pairing of conformal weight $-1$
$$
\langle-,-\rangle : \mathbf{L}^{\alpha}\otimes \mathbf{L}^{1 - \alpha} \to \mathbb{C}.
$$
We define $N = \mathop{\bigoplus}\limits_{\alpha \in \mathbb{Q}}\mathbf{L}^{\alpha}\otimes \omega^{1 - \alpha}_X$. Let $\{x_a\}_{1\leq a \leq n}$ be a basis of $\mathbf{L}$. We consider  $P^{\circ} \subset j_*j^* (N\oplus \mathbf{1}^{\circ})\boxtimes(N\oplus \mathbf{1}^{\circ})$ defined by the following basis
$$
\begin{aligned}
	&\mathbf{1}^\circ\boxtimes\mathbf{1}^\circ,\\
	&\mathbf{1}^{\circ}\boxtimes x_a,\quad  x_a\boxtimes \mathbf{1}^{\circ},\quad 1\leq a\leq n,\\
	&x_a\boxtimes y_b  - \frac{\Omega_{ab}}{z_1 - z_2}\mathbf{1}^{\circ}\boxtimes \mathbf{1}^{\circ},\quad 1\leq a,b\leq n.
\end{aligned}
$$
where $\Omega_{ab} = \langle x_a,x_b\rangle $.
\begin{prop}
	The chiral algebra $\frac{	\mathcal{A}(N\oplus \mathbf{1}^{\circ},P)}{\langle \mathbf{1}^{\circ}_{\omega}-1_{\omega}\rangle}$ defined as above is isomorphic to the chiral Weyl algebra $\EuScript{W}(\mathbf{L},\langle-,-\rangle)$ defined in \cite[Section 3.8.1, pp228]{beilinson2004chiral}. The corresponding vertex algebra is the $\beta\gamma-bc$ system.
\end{prop}
\begin{proof}
	This is a corollary of Proposition \ref{proof_envelop}.
\end{proof}

We assume that the symplectic pairing is non-degenerate. Then the dual relation $P^\perp$ is given by the following basis
$$
\begin{aligned}
	&s^{-1}\mathbf{1}^{\circ}_{\omega^{-1}}\boxtimes s^{-1}\mathbf{1}^{\circ}_{\omega^{-1}} + \sum_{1\leq a,b\leq n}\frac{\Omega_{ab}}{z_1 - z_2}s^{-1}x_a^\vee \boxtimes s^{-1}x^\vee_b,\\
	& s^{-1}\mathbf{1}^{\circ}_{\omega^{-1}}\boxtimes s^{-1}x^\vee_a,\quad s^{-1}x^\vee_a \boxtimes s^{-1}\mathbf{1}^{\circ}_{\omega^{-1}} ,\quad 1\leq a \leq n,\\
	& s^{-1}x_a^\vee \boxtimes s^{-1}x^\vee_b,\quad 1\leq a,b\leq n,
\end{aligned}
$$

The quadratic projection $(\mathbf{q}P^{\circ})^{\perp}$ is given by the following basis
$$
s^{-1}x_a^\vee \boxtimes s^{-1}x^\vee_b.
$$
We get the graded commutative chiral algebra $\mathcal{A}(s^{-1}N^{\vee}_{\omega^{-1}},(\mathbf{q}P^{\circ})^{\perp}) =  \mathrm{Sym}((s^{-1}N^{\vee})_{\mathcal{D}})$. The differential $d$ is zero. The curving element is given as follows
$$
\begin{aligned}
	\iota & = (h\boxtimes \mathrm{id})\mu(s^{-1}\mathbf{1}^{\circ}\boxtimes s^{-1}\mathbf{1}^{\circ})\\
	& = -\sum_{1\leq a,b\leq n} (h\boxtimes \mathrm{id})\mu(\frac{\Omega_{ab}}{z_1 - z_2}s^{-1}x_a^\vee dz_1 \boxtimes s^{-1}x^\vee_bdz_2).
\end{aligned}
$$
By identifying the chiral operation with the commutative product of $\mathrm{Sym}((s^{-1}N^{\vee})_{\mathcal{D}})^l$ as in the proof of  \ref{proof_CE}, we see that the curving $\iota$ is an element of $\mathrm{Sym}((s^{-1}N^{\vee})_{\mathcal{D}})$. Moreover, $\iota$ is given by
$$
	-\sum_{1\leq a,b\leq n}\Omega_{ab}(s^{-1}x_a^\vee \cdot s^{-1}x^\vee_b)dz.
$$

\bibliographystyle{amsplain}
\providecommand{\bysame}{\leavevmode\hbox to3em{\hrulefill}\thinspace}
\providecommand{\MR}{\relax\ifhmode\unskip\space\fi MR }
\providecommand{\MRhref}[2]{%
  \href{http://www.ams.org/mathscinet-getitem?mr=#1}{#2}
}
\providecommand{\href}[2]{#2}

\end{document}